\documentclass[a4paper,11pt,leqno]{amsart}
\usepackage{amsmath}
\usepackage{amsthm}
\usepackage{amsfonts}
\usepackage{amssymb}
\usepackage[mathscr]{eucal}
\usepackage{eucal}
\usepackage{graphicx}
\usepackage{mathrsfs}
\theoremstyle{plain}
\newtheorem{thm}{Theorem}[section]
\newtheorem{lem}[thm]{Lemma}
\newtheorem{prop}[thm]{Proposition}
\newtheorem{cor}[thm]{Corollary}
\newtheorem{rem}[thm]{Remark}

\newtheorem*{thmnn}{Theorem}

\theoremstyle{definition}

\theoremstyle{remark}

\def\l{\lambda}
\def\L{\Lambda}
\def\R{\mathbb{R}}
\def\N{\mathbb{N}}
\def\S{\Sigma}

\def\e{\epsilon}
\def\g{\gamma}
\def\d{\delta}
\def\D{\Delta}
\def\z{\zeta}
\def\G{\Gamma}
\def\a{\alpha}
\newcommand{\dist}{\text{dist}}
\newcommand{\ddist}{\text{dist}^*}

\DeclareMathOperator{\diam}{diam}

\DeclareMathAlphabet{\mathscr}{OT1}{pzc}{m}{it}

\begin{document}

\title{\centerline{Sets of Constant Distance from a Jordan Curve }}
\date{\today}
\author{Vyron Vellis}

\author{Jang-Mei Wu}

\address{Department of Mathematics, University of Illinois,  1409 West Green Street, Urbana, IL 61820, USA}

\email{vellis1@illinois.edu}

\email{wu@math.uiuc.edu}

\thanks{Research supported in part by the NSF grants DMS-0653088 and DMS-1001669.}

\thanks{Figures are illustrated by Julie H.-G. Kaufman.}

\subjclass[2010]{Primary 30C62; Secondary 57N40}
\keywords{chordal property, Jordan curves, distance function, level sets, quasicircles, chord-arc curves}

\begin{abstract}

We study the $\e$-level sets of the signed distance function to a planar Jordan curve $\G$, and ask what properties of $\G$ ensure that  the $\e$-level sets
are  Jordan curves, or uniform quasicircles, or uniform chord-arc curves for \emph{all} sufficiently small $\e$.
Sufficient conditions are given in term of a scaled invariant parameter for measuring the local deviation of subarcs from their chords. The chordal conditions given are sharp.

\end{abstract}

\maketitle

\section{Introduction}\label{intro}

Let $A$ be a compact subset of $\R^2$. For each $\e>0$,  define the $\e$\emph{-boundary} of $A$ to be the set
\[
\partial_{\e}(A) = \{x \in \R^2 \colon \dist(x,A) = \e \}.
\]
Brown showed  in \cite{Brown} that for all but countably many $\e$, every component of $\partial_{\e}(A) $ is a point, a simple arc, or a simple closed curve. In \cite{Ferry}, Ferry showed, among other results, that $ \partial_{\e}(A)$ is a $1$-manifold for almost all $\e$. Fu  \cite{Fu} generalized Ferry's results, and proved that for all $\e$ outside a compact set of zero $1/2$-dimensional Hausdorff measure, $ \partial_{\e}(A)$ is a Lipschitz $1$-manifold. Papers \cite{Ferry} and \cite{Fu} include theorems in higher dimensional Euclidean spaces; the work for dimensions $n\ge 3$ is more demanding.

Let $\G$ be a Jordan curve in $\R^2$ and  $\Omega$ be the bounded component of  $\R^2\setminus \G$. We define the \emph{signed distance function}
\[
\ddist(x,\G)=
\begin{cases}
\dist(x,\G),  & x\in \Omega,\\
-\dist(x,\G),  & x\in \R^2\setminus \Omega;
\end{cases}
\]
and define for any $\e\in (-\infty, \infty)$, the $\e$\emph{-level set} of the signed distance function to be
\[
\g_{\e} = \{ x\in \R^2 \colon \ddist(x,\G) = \e \}.
\]

What properties of $\G$ ensure that  the $\e$-level sets
are  Jordan curves, or uniform quasicircles, or uniform chord-arc curves for \emph{all} $\e$ sufficiently close to $0$?

\begin{figure}[ht]
\includegraphics[scale=0.9]{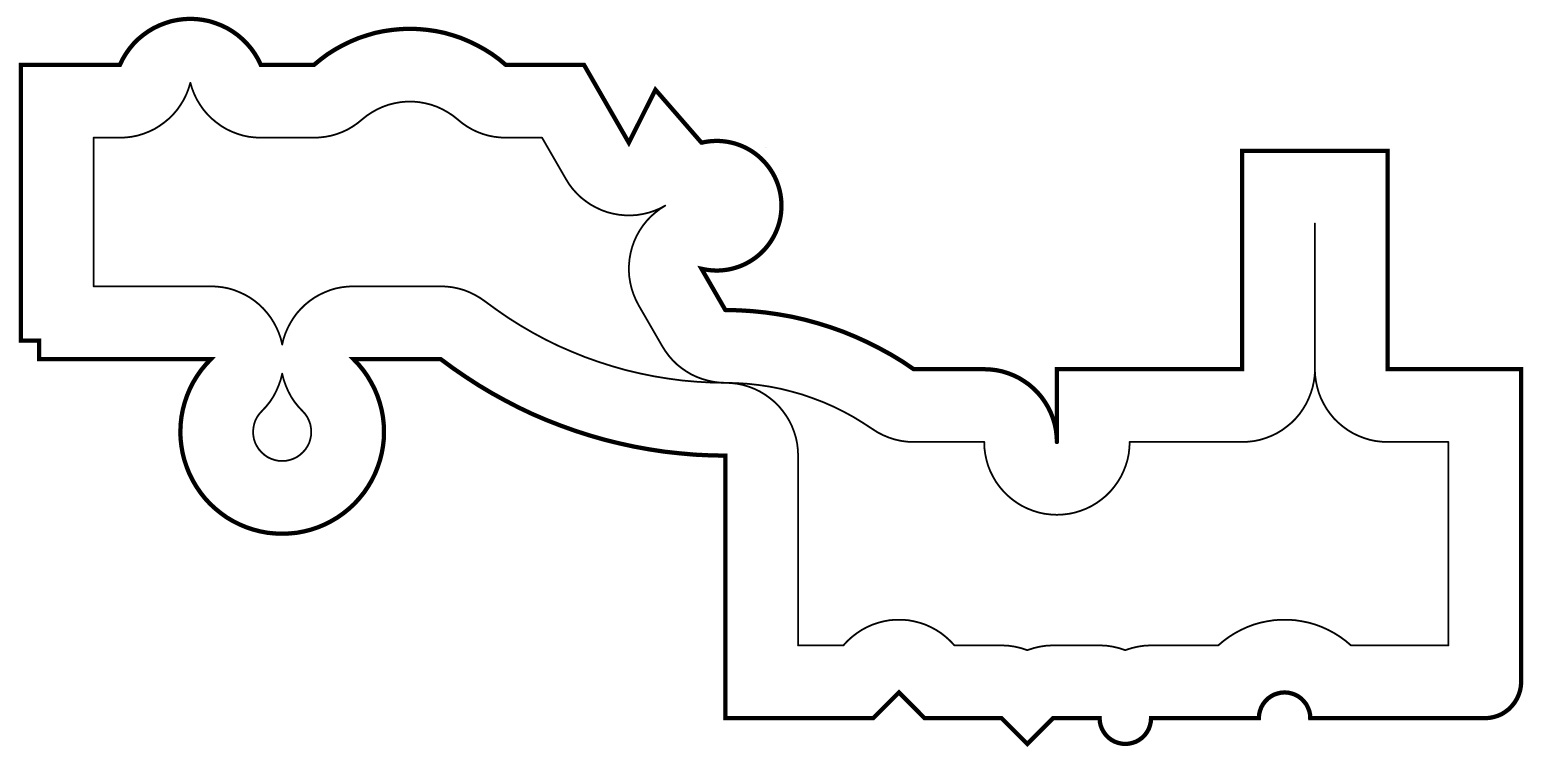}
\caption{A level set of a Jordan curve}
\label{fig:figure2}
\end{figure}

We say that a Jordan curve $\G$ in $\R^2$ has the \emph{level Jordan curve  property} (or LJC \emph{property}), if there exists $\e_0>0$ such that the level set $\g_{\e}$ is a Jordan curve for every $0< |\e| \leq \e_0$. A Jordan curve $\G$ is said to have
the \emph{level quasicircle property} (or LQC \emph{property}), if there exist $\e_0>0$ and $K\geq 1$ such that the level set $\g_{\e}$ is a $K$-quasicircle  for every $0< |\e| \leq \e_0$. Finally, a Jordan curve $\G$ is said to have the \emph{level chord-arc property} (or LCA \emph{property}), if there exist $\e_0>0$ and $C\geq 1$ such that $\g_{\e}$ is a $C$-chord-arc curve  for every $0< |\e| \leq \e_0$.
It is not hard to see that if $\G$ has the LQC property then it is a quasicircle and if $\G$ satisfies the LCA property then it is a chord-arc curve; see Theorem \ref{LCAmain}.

Given two points $x,y $ on a Jordan curve $\g$, we take $\g(x,y)$ to be the subarc of $\g$ connecting $x$ and $y$ that has a smaller diameter,
or, to be either subarc when both  have the same diameter.

Modeled on the \emph{linear approximation property} of Mattila and Vuorinen \cite{MaVu}, we define, for a Jordan curve $\G$ in the plane, a scaled invariant parameter to measure the local deviation of the subarcs from their chords.
For  points $x, y$ on a Jordan curve $\G$ and the infinite line $l_{x,y}$ through $x$ and $y$, we set
\[ \z_{\G} (x,y) = \frac{1}{|x-y|}\sup_{z \in \G(x,y)} \dist (z,l_{x,y}). \]
A Jordan curve $\G$ is said to have the $(\z, r_0)$-\emph{chordal property} for a certain $\z >0$ and $r_0>0$, if
\[
\sup_{x,y \in\G, |x-y|\leq r_0} \z_{\G}(x,y)  \leq \z.
\]
We set
\[ \z_{\G} = \lim_{r_0\to 0}\sup_{x,y \in\G,|x-y|\leq r_0} \z_{\G}(x,y).\]
This notion of chord-likeness provides us a gauge for studying the geometry of level sets.

\begin{thm}\label{betathmLC}
Let $\G$ be a Jordan curve in $\R^2$. If $\G$ has the $(1/2, r_0)$-\emph{chordal property} for some  $r_0>0$, then $\G$ has the level Jordan curve property.
\end{thm}

\begin{thm}\label{betathmLQC}
Let $\G$ be a Jordan curve in $\R^2$. If $\z_\G <1/2$, then  $\G$ has the level quasicircle property. In particular,
if $\G$ has the $(\z, r_0)$-\emph{chordal property} for some $0<\z<1/2$ and  $r_0>0$, then there exist $\e_0>0$ and $K\ge 1$ depending on $\z$, $r_0$ and the diameter of $\G$ so that the level sets $\g_\e$ are $K$-quasicircles for all $0<|\e|<\e_0$.
\end{thm}

Lemmas \ref{4points} and \ref{K-D} lead naturally to the $(1/2,r_0)$-chordal condition for LJC in Theorem \ref{betathmLC};
they show that the behavior of the level set near branch points in Figure \ref{fig:figure2} is, in some sense, typical.

Condition $\z_\G<1/2$ in Theorem \ref{betathmLQC} is used to prove the Ahlfors $2$-point condition for level Jordan curves, thereby establishing the LQC property.

The chordal conditions in both theorems are sharp. The sharpness in Theorem  \ref{betathmLC} is given in Remark \ref{sharpLJC}, and the sharpness in Theorem \ref{betathmLQC} will be illustrated in Remark \ref{sharpLQC}.

\medskip

Moreover, using a lemma of Brown \cite[Lemma 1]{Brown}, we are able to show the following.

\begin{thm}\label{LCAmain}
A Jordan curve $\G$ in the plane satisfies the level chord-arc property if and only if it is a chord-arc curve and has the level quasicircle property.
\end{thm}


\medskip

The paper is organized as follows. We discuss the chordal property in Section \ref{flatness}, and study geometric properties of level sets in Section \ref{geometry}.
In Section \ref{mainresults}, we prove Theorems \ref{betathmLC} and  \ref{betathmLQC}
and give examples to show the sharpness of these theorems. We give the proof of  Theorem \ref{LCAmain} in Section \ref{LCAresults}. Finally in Section \ref{snowflakes}, we provide an additional example based on Rohde's $p$-snowflakes.

\section{Preliminaries}\label{prelim}

A homeomorphism $f\colon D\to D'$ between two domains in $ \mathbb{R}^2$  is called $K$-\emph{quasiconformal}  if it is orientation preserving, belongs to $ W_{loc}^{1,2}(D)$, and satisfies the distortion inequality
\[
|Df(x)|^2 \le K J_f(x) \quad \text{a. e.} \,\,\, x \in D,
\]
where $Df$ is the formal differential matrix and $J_f$ is the Jacobian. The smallest $K=K(f)$ for which the above inequality holds almost everywhere is called the distortion of the mapping $f$.

A Jordan curve $\g$ in $\mathbb{R}^2$ is called a $K$-\emph{quasicircle} if it is the image of the unit circle $\mathbb S^1$ under a $K$-quasiconformal homeomorphism of $\mathbb{R}^2$.
A geometric characterization due to Ahlfors \cite{Ah} states that a Jordan curve $\g$ is a $K$-quasicircle if and only if it satisfies the \emph{2-point condition}:
\begin{equation}\label{3pts}
\text{there exists } C>1 \text{ such that for all }  x,y \in \g, \, \,\diam{\g(x,y)} \leq C|x-y|,
\end{equation}
where the distortion $K$ and the $2$-point constant $C$ are quantitatively related.

A long list of remarkably diverse characterizations of quasicircles has been found. See
the monograph of Gehring \cite{Gehring-characterization} for informative discussion.

A homeomorphism $f\colon D\to D'$ between two domains in $\mathbb{R}^2$ is said to be
$L$-\emph{bi-Lipschitz}, if there exists $L\geq 1$ such that for any $x,y \in D$
\[
\frac{1}{L}|x-y| \leq |f(x)-f(y))| \leq L |x-y|.
 \]
A rectifiable Jordan curve $\g$ in $\mathbb{R}^2$ is called a $C$-\emph{chord-arc curve} if there exists $C\geq 1$ such that  for any $x, y \in \g$, the length of the shorter component, $\g'(x,y)$,  of  $\g \setminus \{x,y\}$ satisfies
\[
\ell(\g'(x,y)) \leq C |x-y|.
\]
Here, and in the future, $\ell (\g)$ denotes the length of a curve $\g$. Every $C$-chord-arc curve is, in fact, the image of $\mathbb S^1$ under an $L$-bi-Lipschitz homeomorphism of $\R^2$, where the constants $C$ and $L$ are quantitatively related; see\cite[p. 23]{Tukia-ext} and \cite[Proposition 1.13]{JeK}.

In the following, we denote by $B(x,r)$  the disk $\{ y \in \mathbb{R}^2 \colon |x-y| < r\}$  and by $S(x,r)$ its boundary $\partial B(x,r)$. In particular, $\mathbb{B}^2 = B(0,1)$ denotes the unit disk and $\mathbb S^{1} = \partial \mathbb{B}^2$ denotes the unit circle.
For $x,y \in \R^2$, denote by $[x,y]$  the line segment having end points $x$ and $y$, by $(x,y)$  the line segment excluding the end points, and by
$l_{x,y}$ the infinite line containing $x$ and $y$.

Finally, we  write $u \lesssim v$ (resp. $u\simeq v$) when $u/v$ is bounded above (resp. above and below) by positive constants.

\section{Chordal Property of Jordan Curves}\label{flatness}

For planar Jordan curves, the connection between the chordal property and the $2$-point condition is easy to establish.

\begin{prop}\label{zetabounded}
A Jordan curve $\G$ is a $K$-quasicircle if and only if
$\G$ is $(\z, r_0)$-chordal for some  $\z >0$ and $r_0 >0$. Constants $K$ and $\z_{\G}$ are quantitatively related, with $\z_{\G}\to 0$ as $K\to 1$.
\end{prop}

The converse of the second statement is not true. Indeed, $\z_{\G}=0$ for every smooth Jordan curve $\G$.

\begin{proof}
Suppose that $\G$ is a $K$-quasicircle and $C$ is the constant in the Ahlfors $2$-point condition (\ref{3pts}) associated to $K$. Then  $\G$ is $(C, \diam \G)$-chordal.

Next suppose  that $\G$ is $(\z, r_0)$-chordal. We claim that $\G$ satisfies property (\ref{3pts}).
Let $ x,y \in \G$.
If $|x-y|\geq  r_0$, then
\[ \diam{\G(x,y)} \leq \frac{\diam{\G}}{r_0}|x-y|. \]
So, we assume $|x-y| < r_0 $, and let $[z,w]$ be the orthogonal projection of $\G(x,y)$ on $l_{x,y}$, with points $z,x,y$ and $w$ listed in their natural order on the line.
In the case that $z \neq x$, choose a point $z'\in \G(x,y)$ whose projection on  $l_{x,y}$ is $z$. Denote by $l$ the line through $x$ and orthogonal to $l_{x,y}$, and fix
a subarc $\sigma$ of $\G(x,y)$ which contains $z'$ and has endpoints, called
$z_1,z_2$, on $\G(x,y) \cap l$. Clearly $\sigma= \G(z_1,z_2)$ and $l=l_{z_1,z_2}$; and by the $(\z, r_0)$-chordal property,
$ \dist(z, l) =\dist(z', l) \leq \z|z_1-z_2| \leq 2\z^2|x-y|.$
It follows that, in all cases,  $|z-w| \leq (4\z^2+1)|x-y|$. Therefore,
\[ \diam{\G(x,y)} \leq (4 \z^2+ 2\z+1) |x-y|. \]
So $\G$ satisfies property (\ref{3pts})  with $C = \max\{4 \z^2+ 2\z+1,\frac{\diam{\G}}{r_0} \}$ and is a $K$-quasicircle for some $K$ depending on $\z,r_0$ and $\diam{\G}$.

The claim that $\z_{\G}\to 0$ as $K\to 1$ follows from a lemma of Gehring \cite[Lemma 7]{Gehring-spirals}, which states that for each $\eta>0$, there exists $K_0=K_0(\eta)>1$ such that if $g$ is a $K$-quasiconformal mapping of $\R^2$ with $K\leq K_0$, and if $g$ fixes two points $z_1$ and $z_2$, then
\[
|g(z)-z|\leq \eta |z_1-z_2|,  \quad  \text{when}\,\, |z-z_1|<|z_1-z_2|.
\]
Quasiconformality in \cite{Gehring-spirals} is defined  using the conformal modulus of curve families, which is quantitatively equivalent to the notion of quasiconformality given in Section \ref{prelim} (See \cite[Theorem 32.3]{Vais1}). This line of reasoning has been used by Mattila and Vuorinen in \cite[Theorem 5.2]{MaVu}.
\end{proof}

By Proposition \ref{zetabounded}, the following will be a  corollary to Theorem \ref{betathmLQC}.

\begin{cor}
There exists a constant $K_0 > 1$ such that all  $ K_0$-quasicircles have the \emph{LQC} property.
\end{cor}

\medskip

Mattila and Vuorinen \cite{MaVu} introduced the \emph{linear approximation property} to study geometric properties of $K$-quasispheres with $K$ close to $1$. Let $k \in \{1,2,\dots,n-1\}$,  $0 \leq \d < 1$, and $r_0>0$.  A set $Z$ in $\R^n$ satisfies a  $(k, \d,r_0)$-linear approximation property  if for each $x \in Z$ and each $0<r<r_0$ there exists an affine $k$-plane $P$  through $x$ such that
\[\dist(z,P) \le \d r\quad \text{ for all}\,\, z\in Z \cap B(x,r).\]
In the same year, Jones \cite{Jo} introduced a parameter, now known as the \emph{Jones beta number}, to measure
the oscillation of a set at all points and in all scales,
for the investigation of the "traveling salesman problem". Later, beta number has been used by Bishop and Jones to study problems on harmonic measures and Kleinian groups.
As it turns out, the Jones beta number and the $\d$-parameter of Mattila and Vuorinen are essentially equivalent.

For planar quasicircles, the chordal property and the linear approximation property are quantitatively related as follows.

\begin{lem}\label{flatequiv}
Let $\G$ be a Jordan curve in $\R^2$.
If  $\G$ has  the  $(\z, r_0)$-chordal property for some $0<\z <1/4$, then it is
$(1, 4\z, r_1)$-linearly approximable, where $r_1= \min\{\frac{r_0}{2}, \frac{\diam \G}{C}\}$ and $C=C(\z,r_0,\diam \G)>1$ is a constant.
On the other hand, if $\G$ is a $K$-quasicircle that has the $(1, \d,r_0)$-linear approximation property, then it is $(C'^2\d, r_0/C')$-chordal, for some constant $C'=C'(K)>1$.
\end{lem}

\begin{proof}
Suppose that $\G$ is  $(\z, r_0)$-chordal. Then $\G$ is a $K$-quasicircle by Proposition \ref{zetabounded}, hence satisfies
the $2$-point condition (\ref{3pts}) for some $C>1$; here $K$ and $C$ depend on $\z, r_0$ and $\diam \G$.
Let $0< r < \min\{\frac{r_0}{2}, \frac{\diam \G}{6 C}\}$; take $x \in \G$, and $x' \in \G\setminus B(x,r)$ such that $|x-x'| \geq \frac{\diam{\G}}{2}$. Let $x_1,x_2$ be the  points in $\G \cap S(x,r)$ with the property that one of the subarcs $\G\setminus \{x_1,x_2\}$ contains $x'$ and lies entirely outside $\overline{B}(x,r)$, and the other subarc, called $\tau$, contains $x$. Since $\diam(\G\setminus \tau) \geq |x-x'| - r \geq \frac{\diam{\G}}{2} - r > 2 C r \geq C |x_1-x_2|$,  we have $\diam \tau \leq C|x_1-x_2|$ and
 $\G(x_1,x_2) = \tau$.
Trivially, $\dist(x,l_{x_1,x_2}) \leq 2\z r$.
 Then,  for $y\in \G(x_1,x_2)$ and the line $l$ through $x$ and parallel to $l_{x_1,x_2}$, we have
\[ \dist(y,l) \leq \dist(y, l_{x_1,x_2}) + \dist( l_{x_1,x_2},l) \leq \z|x_1-x_2|+2\z r \leq 4 \z r;\]
and the first claim follows.

For the second claim, suppose that $\G$ is a $K$-quasicircle that has the $(1, \d,r_0)$-linear approximation property. So $\G$ satisfies
the $2$-point condition (\ref{3pts}) for some $C=C(K)\ge 1$.
Take $x,y \in \G$ with $0< |x-y | < \frac{r_0}{4C}$ and  $r = (C+1)|x-y|$, then
 $\G(x,y) \subset B(x,r)$.  Since $r <r_0$,
 there exists a line $l$ containing $x$ such that
\[ \G(x,y) \subset \{ z \in B(x,r) \colon \dist(z,l) \leq \d r \}. \]
In particular, $\dist(y,l)\leq \d r$.
Given $z\in \G(x,y)$, take a point $z' \in l\cap B(x,r)$ with $|z-z'|\le \d r$; then, from elementary geometry we get
\[
\dist(z', l_{x,y}) \leq \frac{\dist(y,l)}{|x-y|}r  \leq (C+1)\d r.
\]
It follows that
$\dist(z, l_{x,y}) \leq |z-z'|+ \dist(z', l_{x,y}) \leq  (C+2)\d r = (C+2)(C+1) \d |x-y|$.
Hence, $\z(x,y) < 6 C^2 \d$,
and the second claim is proved.
\end{proof}

\section{Geometry of Level Sets}\label{geometry}

Let $\G$ be a Jordan curve in $\R^2$, $\Omega$ be the bounded component of  $\R^2\setminus \G$.
For any $\e\neq 0$,  define the open set

\[
\D_{\e} =
\begin{cases}
\, \{ x\in \R^2 \colon \ddist(x,\G) > \e \}, & \e>0,\\
\, \{ x\in \R^2 \colon \ddist(x,\G) < \e \}, & \e<0.
\end{cases}
\]
In general, $\D_{\e}$ need not be connected, and $\overline{\D_\e}$ and $\D_\e \cup \g_\e$ may not be equal (see Figure \ref{fig:figure2}).

\medskip

However, \emph{for any $\e > 0$, the sets $\overline{\Omega} \setminus \D_{\e}$, $(\R^2\setminus \Omega)\setminus \D_{-\e}$, and $\R^2\setminus (\D_{\e}\cup\D_{-\e})$ are path-connected.} Indeed, given $x,y \in \overline{\Omega} \setminus \D_{\e}$,  take $x',y' \in \G$ such that $|x-x'| = \dist(x,\G)$ and $|y-y'| = \dist(y,\G)$. Note  that $[x,x']$ and $[y,y']$ are entirely in $\overline{\Omega} \setminus \D_{\e}$. So $x,y$ can be joined in $\overline{\Omega} \setminus \D_{\e}$ by the arc $[x,x'] \cup \G(x',y') \cup [y,y']$. Path-connectedness of  the other two sets  can be proved analogously.

\medskip
Furthermore,
\emph{given  $x, y \in \Omega$, let  $x', y'$ be  points in  $\G$ with the property that
$|x-x'|=\emph{\dist}(x,\G)$ and $|y-y'|= \emph{\dist}(y,\G)$. If $x,y,x',y'$ are not collinear
then the segments $[x,x']$, $[y,y']$ do not intersect except perhaps at their endpoints.}
Indeed, if there is a point $z \in [x,x']\cap[y,y']$  which is not $ x'$ or $ y'$, then $|z-x'|= |z-y'|$.
By the triangle inequality,
\[
|x-y'| < |x-z|+ |z-y'| = |x-z|+ |z-x'| = \dist(x, \G),
\]
which is a contradiction. The same claim is true when $x,y \in \R^2 \setminus \overline \Omega$. The non-crossing property of $[x,x'],[y,y']$ is a special case of Monge's observation on optimal transportation; see \cite[p. 163]{Villani}.
\bigskip

In the following, a point will be considered as a degenerate arc.

Given
a closed subset $\L$ of $\G$ and a number $\e \neq 0$, we define
\[\g_{\e}^{\L}= \{ x \in \g_{\e} \colon \dist(x,\L) = |\e| \}.\]
In general, the set $\g_{\e}^{\L}$ may be empty even when $\L$ is a non-trivial arc (see Figure \ref{fig:figure2}).
However, $\g_{\e}^{\L}$ is an arc when  $\g_\e$ is a Jordan curve and $\L$ is connected, as we see from the following lemma.

\begin{lem}\label{orientation}
Let $\G$ be a Jordan curve in $\R^2$, and assume that  for some $\e\neq 0$, the level set $\g_{\e}$ is a Jordan curve. If $\L$ is a closed subarc of $\G$ and $\g_{\e}^{\L}$ is nonempty, then $\g_{\e}^{\L}$ is a subarc of $\g_{\e}$.
\end{lem}

\begin{proof}
It suffices to prove that if $x$ and $y$ are two distinct points in $\g_{\e}^{\L}$ then, one of the two subarcs $\l_1,\l_2$ of $\g_{\e}$ connecting $x$ and $y$ is entirely in $\g_{\e}^{\L}$.

Assume first that $\e>0$. We claim that if $\l_1 \setminus\g_{\e}^{\L} \neq \emptyset$ then $\l_2 \subset \g_{\e}^{\L}$.
Take $ z \in \l_1 \setminus\g_{\e}^{\L}$, $x',y' \in \L$ and $z' \in \G\setminus \L$ such that
\[ |x-x'| = |y-y'| = |z-z'| = \e, \]
and let $\L_1$ be the subarc of $\L$ that joins $x',y'$ ($\L_1$ could be just a point). We  know that the open line segments $(x,x'),(y,y'),(z,z')$ and the Jordan curve $\g_{\e}$ do not intersect one another.
Let $U_1$ be the quadrilateral (possibly degenerated in the case $x'=y'$) enclosed by the Jordan curve $[x,x']\cup \l_1 \cup [y,y'] \cup \L_1$. Then the open arc $\l_2\setminus \{x,y\}$ must be contained in  $U_1$. For otherwise, $\l_2\setminus \{x,y\}$ would intersect either the arc $[x,x']\cup \l_1 \cup [y,y'] $ or the segment $ [z,z']$; this is impossible in view of properties of the distance function $\dist(\cdot,\G)$. Therefore, the quadrilateral $U_2$ enclosed by the Jordan curve $[x,x']\cup \l_2 \cup [y,y'] \cup \L_1$ is contained in $U_1$.
Suppose now that $\l_2 \subset \g_{\e}^{\L}$ is false. Then, by the argument above with the roles of $\l_1$  and  $\l_2$ reversed, we get $U_1 \subset U_2$. Hence $U_1=U_2$, which is impossible. This proves the claim.

When $\e<0$, we choose $J_1 = [x,x']\cup \l_1\cup [y,y'] \cup \L_1$ and $J_2 = [x,x']\cup \l_2 \cup [y,y'] \cup \L_1$ as before, however define $U_1$ and $U_2$ to be the unbounded components of $\R^2 \setminus J_1$  and $\R^2 \setminus J_2$, respectively, then proceed as above.
\end{proof}

We show next that when two points $x,y$ on a level Jordan curve $\g_\e$ have a common closest point on $\G$, $\g_\e(x,y)$ is a circular arc.

\begin{lem}\label{circgamma}Let $\G$ be a Jordan curve in $\R^2$, and assume that the level set $\g_{\e}$ is a Jordan curve for some $\e \neq 0$. Suppose  that there exist $x,y \in \g_{\e}$ and $ z\in \G$ such that $|x-z| = |y-z| = |\e|$. Then $\g_{\e}(x,y)$ is a circular arc on $S(z,|\e|)$ of length at most $\pi |\e|$.
\end{lem}

\begin{proof}
By Lemma \ref{orientation}, $ \g_{\e}^{\{z\}}=\{w\in \g_\e\colon |w-z|=\e\}$ is a subarc of $\g_\e \cap S(z,|\e|)$.
Since $\{x,y\}\subset \g_{\e}^{\{z\}}$,
$\g_{\e}(x,y)$ is one of the two subarcs  of $S(z,|\e|)$ that connects $x$ and $y$.

Suppose that  $\ell(\g_{\e}(x,y))>\pi |\e|$.  Then
the domain $U$ enclosed by the Jordan curve $\g_{\e}(x,y)\cup [x,y]$ contains precisely one point from $ \G$, namely the point $z$;  all other points on $\G$ are  in the exterior of $U$. So, $\G$ intersects the segment $[x,y]$; consequently, $\dist(x,\G)<|\e|$ and $ \dist(y,\G) < |\e|$. This is a contradiction.
\end{proof}

The following lemma shows that components of  $\D_\e$ satisfy a weak form of one part of  \emph{linearly local connectedness} introduced in \cite{Gehring-Vaisala}; see also \cite[p. 67]{Gehring-ICM}.
In particular, \emph{$\D_\e$ has no inward cusps.}

\begin{lem}\label{curvesincomp}
Let $\G$ be a Jordan curve, $\e\neq 0$ and $D$ a connected component of $\D_{\e}$. Then for any $x,y \in D$ with $|x-y|\leq 2|\e|$, there exists a polygonal arc $\tau$ in $D$ that joins $x$ and $y$ and has  $\diam{\tau} \leq 5|x-y|$.
\end{lem}

\begin{proof}
Suppose first that $\e>0$.
Let $x, y $ be two points in $ D$ with $|x-y|\leq2\e$. So $[x,y]\cap \G = \emptyset$ and the segment $[x,y]$ is contained in the bounded component $\Omega$ of $\R^2\setminus \G$.
Let $\tau'$ be any curve in $D$ that connects $x$ to $y$. After approximating $\tau'$ by a polygonal curve, erasing the loops and making small adjustments near the segment $[x,y]$, we may assume that $\tau'$ is a simple polygonal curve which intersects $[x,y]$ in a finite set. In other words, $\tau'$ is the union of finitely many simple polygonal subarcs $\sigma'$ in $D$, each of which meets $[x,y]$ precisely at its end points. The curve $\tau$ in the proposition will be obtained by replacing
each $\sigma'$ with a polygonal arc $\sigma$ in $D \cap B(x,\frac{5}{2}|x-y|) $ with the same  end points.

Fix such a subarc $\sigma'$ having end points $a, b\in [x,y]$. Assume that $\sigma' \setminus \overline{B}(x,2|x-y|) \neq \emptyset $; otherwise, just let  $\sigma=\sigma'$.
Let $U$ be the domain enclosed by the Jordan curve
$\sigma' \cup [a,b]$. Since $\partial U \cap \G=\emptyset$ and $\Omega$ is simply connected, $\overline U \subset \Omega$.
We claim that
\[ U \setminus {B(x,2|x-y|)} \subset \D_{\e}. \]
Otherwise, take a point  $z\in U \setminus (B(x,2|x-y|) \cup  \D_{\e})$ and assume as we may, by the continuity of the distance function, that $z\in \g_\e$. Let $z'$ be a point on $\G$ for which $|z-z'|=\dist (z,\G)=\e$. Since $\overline U \subset \Omega$, $z' \notin \overline{U}$ and the open segment $(z,z')$ intersects $\partial U$ at some point $z''$. If $z''$ is in $ [a,b]\subset [x,y]$ then
\[
\dist(x,\G)\leq |x-z'| \leq |x-z''|+|z''-z'|= |x-z''|+ \e -|z-z''|<\e,
\]
a contradiction.
If $z''$ is in $\sigma'$ then
$ \e =|z-z'|> |z''-z'|\ge \dist(z'',\G)> \e$, again a contradiction. This proves the claim.

Let $U'$ be the connected component of $U \cap B(x,2|x-y|)$ that contains the segment $[a,b]$ in its boundary. Since $U$ is a polygon, $U'$ is simply connected and $\partial U'$ is a Jordan curve.
In particular, $\partial U'\setminus (a,b)$ is composed of finitely many line segments in $D$ and finitely many subarcs of $S(x,2|x-y|)$. In view of the claim above,
$\partial U' \setminus (a,b)$ is  an arc contained in $D$.
After replacing each maximal circular subarc of  $\partial U' \setminus (a,b)$ by a polygonal arc nearby, we obtain a polygonal arc $\sigma$ in $ D\cap B(x,\frac{5}{2}|x-y|)$ connecting $a$ to $b$. The arc $\tau$ in the proposition is given by the union of these new  $\sigma$'s.

Assume next that $\e<0$. Given two points $x,y\in D$ with $|x-y|\leq 2|\e|$, we define $\tau',\sigma', a, b, U$ as before, and need to replace each subarc $\sigma'$ of $\tau'$ by a new arc $\sigma$ in $ D\cap B(x,\frac{5}{2}|x-y|)$.
Since $\partial U =\sigma'\cup [a,b]\subset \D_\e$ is contained   in the unbounded component of $\R^2\setminus \G$ and $\partial U$ contains some points in $\D_\e$,
we have either $\overline \Omega \cap \overline U  =\emptyset$, or $\overline \Omega \subset U$.

Suppose $\overline \Omega \cap \overline U =\emptyset$. Then
$\overline U \subset \R^2 \setminus \overline {\Omega}$. We first check that $U \setminus B(x,2|x-y|) \subset \D_{\e}$, then choose a replacement  $\sigma$ for $\sigma'$ following the same
steps as in the case $\e>0$.

Suppose  $\overline{\Omega} \subset U$. We set $V = \R^2 \setminus \overline{U}$, then prove $V \setminus B(x,2|x-y|) \subset \D_{\e}$ by replacing $U$ with $V$ in the argument for the case $\e>0$.
Define $V'$ to be the  component of $V \cap B(x,2|x-y|)$ that contains $[a,b]$ in its boundary, and choose $\sigma$ to be a polygonal curve close to $\partial V'\setminus (a,b)$  in $ D\cap B(x,\frac{5}{2}|x-y|)$ and having end points $a, b$.
\end{proof}

We next prove that the boundary of any connected component of $\D_{\e}$ is a Jordan curve. We will need a theorem of Lennes in \cite{Len} which gives a sufficient condition for the frontier, of  a bounded planar domain, to be a Jordan curve. Let $D$ be a bounded domain
and $p$ a closed polygonal curve which encloses $\overline{D}$ in its interior.  Let $E'$  be the set of all points in the plane that can be joined to  $p$ by a continuous curve in the complement of $D$. The \emph{frontier} $F$ of $D$ is the set of all common limit points of $E'$ and $D$, that is, $F=\overline E' \cap \overline D$. Define moreover the \emph{interior set of the frontier} $F$ to be
$I = \R^2 \setminus (E' \cup F)$ and the \emph{exterior set of the frontier} $F$ to be $E = E' \setminus F$. Observe that all the above definitions are independent of the choice of $p$.

Furthermore, a point  $x \in F$ is said to be \emph{externally accessible} if there exists a finite or a continuous infinite polygonal path $\tau \colon [0,1] \to \R^2$ such that $\tau([0,1))\subset E$ and $\tau(1) = x$. And a point  $x \in F$ is said to be \emph{internally accessible} if there exists a finite or a continuous infinite polygonal path $\tau \colon [0,1] \to \R^2$ such that $\tau([0,1))\subset I$ and $\tau(1) = x$. Lennes proved the following.

\begin{lem}[{\cite[Theorem 5.3]{Len}}]
If every point of a frontier $F$ possesses both the internal and the external accessibility, then $F$ is a Jordan curve.
\end{lem}

We now apply the theorem of Lennes to prove the following.

\begin{lem}\label{appr}
Let $\G$ be a Jordan curve and $\e \neq 0$. Then, the boundary of every  connected component of $\D_{\e}$ is a Jordan curve.
\end{lem}

\begin{proof}
We prove the claim for $\e>0$ only. The proof for the case $\e<0$ is essentially the same. Recall that $\Omega$ is the bounded component of  $\R^2\setminus\G$. Let $D$ be a connected component of $\D_\e$, and $p$ be a closed polygonal curve that encloses $\overline{\Omega}$ in its interior. Every point $x\in \Omega \setminus D$ can be joined to one of its closest points on $\G$ by a line segment  entirely outside $D$, then to $p$ by a curve in $\R^2\setminus \Omega$. Therefore,  $E'= \R^2 \setminus D$, $F= \overline E' \cap \overline D= \partial D$ and $I=D$,   and  any point in $\partial D$ is externally accessible.

To check the internal accessibility, we take $x \in \partial D$, and a sequence $\{x_n\} $ in $D$ with distance $|x_n - x|<2^{-n}\e$ for every $n \ge 1$. By Lemma \ref{curvesincomp}, there exist a family of polygonal arcs $\{\tau_n\}_{n\in \N}$ in $D$ such that $\tau_n$ joins  $x_n$ to $x_{n+1}$ and has
$\diam{\tau_n} \leq 5|x_n - x_{n+1}| \leq 2^{1-n}\e$.
Then, take $\tau$ to be the infinite polygonal path $\{x\} \cup \bigcup_{n\ge 1}\tau_n$. This proves that $x$ is internally accessible, and by Lennes' theorem, we conclude that $\partial D$ is a Jordan curve.
\end{proof}

Components of $\D_\e$ satisfy one of the two conditions for the linearly local connectedness \cite{Gehring-ICM}, when $\G$ is a quasicircle.

\begin{lem}\label{quasiconvex}
Suppose that $\G$ is a $K$-quasicircle. Then, there exists a constant $M>0$ depending only on $K$ such that for any $\e\neq 0$, for any connected component $D$ of $\D_{\e}$, and for any two points $x,y \in D$, there exists a curve $\tau$ in $ D$ joining $x$ and $ y$ such that $\diam{\tau} \leq M|x-y|$.
\end{lem}

\begin{proof}

 In view of Lemma \ref{curvesincomp}, we consider points
$x$ and $y$ in $D$ with distance $|x-y|>2|\e|$ only. The proof follows closely that of Lemma \ref{curvesincomp}; however, the segment $[x,y]$ here may intersect $\G$.

Assume first that $\e>0$.  Since $\G$ is a $K$-quasicircle, it satisfies condition (\ref{3pts}) for some $C=C(K)>1$.
Fix  a simple polygonal curve $\tau'$ in $D$ joining $x$ and $y$ that intersects $[x,y]$ in a finite set. As in Lemma \ref{curvesincomp}, we will replace each
subarc $\sigma'$ of $\tau'$ that has end points in $[x,y]$ and does not intersect $[x,y]$ anywhere else, by a new
arc $\sigma$ in $ D \cap B(x,(C+2)|x-y|) $ having the same end points.

Fix such a subarc $\sigma'$ having end points $a, b \in [x,y]$. Assume that $\sigma' \setminus \overline{B}(x,(C+2)|x-y|)\neq \emptyset$; otherwise, just set $\sigma=\sigma'$.
The domain $U$ enclosed by the Jordan curve
$\sigma' \cup [a,b]$ may now contain points outside $\Omega$.
We claim nevertheless that
\begin{equation}\label{U}
U \setminus {B(x,(C+2)|x-y|)} \subset \D_{\e}.
\end{equation}
Suppose $U \setminus {B(x,(C+2)|x-y|)} \nsubseteq \D_{\e}$. As before, we may pick a point  $z\in (U \setminus {B(x,(C+2)|x-y|)}) \cap \g_{\e}$ and a point $z'\in \G$ such that  $|z-z'|=\dist (z,\G)=\e$. Suppose $z'\notin U$;  the segment $[z,z']$ must intersect $\partial U$ at some point $z''$. Because $|x-y|>2|\e|$,  the point $z''$ cannot be in $ [x,y]$, therefore $z'' \in \sigma'$. Hence
$ \e =|z-z'|> |z''-z'|\ge \dist(z'',\G)> \e$, a contradiction. So  $z'$ must be in $U$, therefore $z'\in (U\cap \G) \setminus B(x,(C+1)|x-y|)$.

Since $\e>0$, $\G$ cannot be entirely in $U$, so $\G\cap \partial U \neq \emptyset$. Since $\partial U=\sigma' \cup [a,b]$ and $\sigma' \subset D$, $\G \cap [a,b]\neq \emptyset$.
Let $z_1,z_2$ be the points in $[a,b]\cap \G$ with the property that the open subarc $\G'$ of $\G$ connecting $z_1$ to $z_2$ and containing the point $z'$, is  entirely in $U$. So $|z_1-z_2| < |a-b|\le |x-y|$ and
\[
\diam{\G'} \geq \dist(z',[z_1,z_2]) \geq |z'-x| -|x-y| \geq C|x-y| > C|z_1-z_2|.
\]
From the  2-point condition \eqref{3pts} it follows that the diameter of the subarc $\G''= \G \setminus \G'$ is at most $ C|z_1-z_2|$. Therefore, $\G'' \subset B(x, (C+1)|x-y|)$, and $\G \subset U \cup  B(x, (C+1)|x-y|)$.

Let $w$ be one of the points on $\sigma'$ that is furthest from $x$.
Since $\G'\setminus \{z_1,z_2\}$ is contained in the open set $U$, $|x-w|> \max_{u \in\G'} |x-u|;$
furthermore
$|x-w| \geq (C+2)|x-y|> \max_{u \in\G''} |x-u|.$
As a consequence,  $w$, also $\sigma'$, is contained in the unbounded component of $\R^2\setminus \G$. This is impossible because $\sigma'\subset D \subset \D_\e \subset \Omega$. Claim (\ref{U}) is proved.

Let $U'$ be the component of  $U \cap B(x,(C+2)|x-y|)$ whose boundary contains $[a,b]$. As in Lemma \ref{curvesincomp}, $\sigma'$ will be replaced by the subarc
$\sigma = \partial U' \setminus (a,b) \subset D\subset \D_{\e}$.
The curve $\tau$ in the proposition is the union of these new  $\sigma$'s.
\medskip

Now suppose $\e<0$. Given $x, y\in D$ with  $|x-y|>2|\e|$,  we define $\tau'$, $\sigma', a, b$, $U$ as before, and need to replace each subarc $\sigma'$ of $\tau'$ by a new arc $\sigma$ in $D\cap B(x,(C+2)|x-y|)$.  When $\e$ is negative, depending on the pair $x,y$, either possibility $U \setminus {B(x,(C+2)|x-y|)} \subset \D_{\e}$ or $U \setminus {B(x,(C+2)|x-y|)} \nsubseteq \D_{\e}$ may actually occur.

Suppose $U \setminus {B(x,(C+2)|x-y|)} \subset \D_{\e}$. Choose $U'$ as in the case $\e>0$ and replace $\sigma'$ by $\sigma=\partial U'\setminus (a,b)$.

Suppose  $U \setminus {B(x,(C+2)|x-y|)} \nsubseteq \D_{\e}$. We follow the argument for claim \eqref{U}, to conclude that $\G\subset U\cup B(x,(C+1)|x-y|)$, therefore $\overline \Omega \subset U \cup B(x,(C+1)|x-y|)$.
Let $U''$ be the component of $B(x,(C+2)|x-y|) \setminus \overline{U}$ that contains $[a,b]$ in its boundary. Note that  $\sigma = \partial U'' \setminus (a,b)$ is a Jordan arc
composed of a finite number of line segments belonging to $\sigma'$ and a finite number of subarcs of $S(x,(C+2)|x-y|)\setminus U$ and that
 $\sigma \subset B(x,(C+2)|x-y|)$.

 It remains to check that
\[S(x,(C+2)|x-y|)\setminus \overline U\subset \D_\e.
\]
Take $z\in S(x,(C+2)|x-y|)\setminus \overline U$ and $z'\in \G$ so that $|z-z'|=\dist(z,\G)$. Write $\G=(\G \cap U) \cup (\G \setminus  U) $.
If $z'\in \G \cap U$, then the segment $(z,z')$ intersects $\partial U=[a,b]\cup \sigma'$ at some point, say $z''$. Since $\dist(z,[a,b]) > |\e|$, the point $z''$ is not in $[a,b]$, hence in  $\sigma'\subset \D_\e$. Therefore,
\[\dist(z,\G\cap U)\geq \dist(z,\G)=|z-z'|=|z-z''|+|z''-z'|\geq \dist(z'',\G)>|\e|.\]
Since $\G \setminus  U\subset B(x, (C+1)|x-y|$, we also have $\dist(z,\G \setminus U)\geq |x-y|> |\e|$.
So $\dist(z,\G)>|\e|$.
This proves the claim and the lemma.
\end{proof}

\begin{rem}\label{remappr}
Both Lemmas \ref{curvesincomp} and \ref{quasiconvex} can be strengthened to include the case when $x$ and $y$ are in $\overline D$. In such case, curves $\tau$ satisfying  the diameter estimates in the lemmas are contained in $D$ with the exception of their endpoints.
\end{rem}

We now state an elementary geometric fact needed in the following two lemmas.
Given $0<\d <\e$ and a point $a=\d e^{i\a}$ in $B(0,\e)$, then
\[S(a,\e)\setminus B(0,\e)= \{a+\e e^{i\theta}\colon |\theta-\a|\leq \pi- \cos^{-1}(\frac{\d}{2\e}) \},
\]
and the circular arc is contained in the sectorial region $\{z\in \R^2 \colon |\arg z -\a| \leq
\cos^{-1}(\frac{\d}{2\e})\}$.
\medskip

Given any $x_0\in \g_\e$, $\e\neq 0$, set
\[\G^{\{x_0\}} = \{ y \in \G \colon |x_0 - y| = |\e| \}.\]

\begin{lem}\label{X}
Suppose $\e \neq 0$ and $x_0$ is a non-isolated point in $\g_\e$. Then the set
$\G^{\{x_0\}}$
lies entirely in a semi-circular subarc of $S(x_0,|\e|)$.
\end{lem}

\begin{proof}
Suppose that $\e>0$.  Choose a sequence of points $a_n$ on $\g_\e$ that converges to $x_0$; set $\d_n=|a_n-x_0|$ and assume as we may that  $0<\d_n <\e$ . Since $\dist(a_n,\G)=\e$, we have $\dist(a_n, \G^{\{x_0\}})\geq \e$, for all $n\geq 1$. In particular $\G^{\{x_0\}}$ is contained in the part of  $S(x_0,\e)$ that is outside $B(a_n,\e)$, which is a circular arc of  length  $(2\pi- 2\cos^{-1}(\frac{\d_n}{2\e}))\,\e$.
The claim follows by letting $n\to \infty$. The proof for the case $\e<0$ is the same.
\end{proof}

Fix a non-isolated point $x_0$ on $\g_\e$, we examine the geometry of the level set $\g_\e$ near $x_0$.

Let $X = \G^{\{x_0\}} $.
Since $X$ is compact, there exist $x_1,x_2 \in X$ (possibly $x_1=x_2$) such that $|x_1-x_2| = \diam{X}\leq 2 |\e|$; and by Lemma \ref{X}, $X$ lies in a subarc $\Sigma$ (possibly degenerated) of $S(x_0,|\e|)$ having endpoints $x_1$ and $x_2$ and of length at most $\pi |\e|$.
Let $U = B(x_1,|\e|) \cup B(x_2,|\e|)$;  clearly $\g_\e \cap U =\emptyset $.
Set $\e_0=(\e^2-|x_1-x_2|^2/4)^{1/2}$. For  $0< \d<\e_0$, the set $S(0,\d) \setminus U$ is a connected arc when $|x_1-x_2|< 2| \e|$, and it has  two components when $|x_1-x_2| = 2| \e|$. Let  $S_\d$ be a component of $S(0,\d) \setminus U$.

\begin{lem}\label{4points}
Suppose $\e \neq 0$ and $x_0$ is a non-isolated point in $\g_\e$. There exists $\d_0 \in (0,\e_0)$ such that
if $0<\d<\d_0$ then the set $\g_\e \cap S_\d$ contains at most two points.

Specifically, if $0<\d<\d_0$, $a\in \g_\e \cap S_\d$, and $a'$ is a point in $\G$ with $|a-a'|=\e$, then at least one of the two components (maybe empty) $S_{\d,a}^1,S_{\d,a}^2$ of $S_\d \setminus \{a\}$ is contained entirely in the disk $B(a',\e)$; in other words,
there exists $j \in \{1,2\}$ such that every point in $S_{\d,a}^j$ has distance strictly less than $\e$ from $\G$.

\end{lem}

\begin{rem}
For every non-isolated point $x_0$ on $\g_\e$ and every $\d \in (0, \e_0)$, the set $\g_{\e}\cap S(x_0,\d)$ contains
at most two points when $|x_1-x_2| < 2| \e|$, and at most four points when $|x_1-x_2| = 2| \e|$. See Figure \ref{fig:figure2} for some of the possibilities.
\end{rem}

\begin{proof}
We prove the lemma for $\e>0$. The  case $\e<0$ is essentially the same.

Assume as we may that $x_0=0$, $x_1 = \e e^{i\tau}, x_2 = \e e^{i(2\pi-\tau)}$ with $\tau \in [\pi/2,\pi]$,  and that $\S \subset \{z \colon \text{Re}\, z \leq 0\}$.  Consider from now on only those $\d$ in $(0,\e_0)$.
Consequently,  $S_\d \subset \{z \colon \text{Re}\, z > 0\}$ when $0<|x_1-x_2|< 2 \e$ and $S(0,\d)\cap \{z \colon \text{Re}\, z \geq 0\} \subset S_\d $ when $x_1=x_2=-\e$;
assume therefore without loss of generality that $S_\d \subset \{z \colon \text{Re}\, z \geq 0\}$  when $|x_1-x_2|= 2 \e$.
It is straightforward to check that
\begin{equation}\label{argument-S}
-(\tau -\cos^{-1}(\frac{\d}{2\e}))\leq \arg z \leq    \tau -\cos^{-1}(\frac{\d}{2\e}) \quad \text{for all} \,\,\, z\in S_\d.
\end{equation}
Fix  a number $\xi \in (0,\frac{\pi}{24})$ depending on $\tau$ so that $\tau -\xi >\pi/2$ when $\tau >\pi/2$, and $\xi=\pi/48$ when $\tau=\pi/2$.
Fix also a number $\d_0 \in (0,\e_0)$, satisfying $\cos^{-1}(\frac{\d_0}{2\e})>\frac{5\pi}{12}$, and having the property that
for any $a\in \g_\e \cap B(0,\d_0)$ and any point $a'$ on $\G$ nearest to $a$, i.e., $|a-a'|=\e$, we have
\begin{equation}\label{argument-a'-1}
\tau -\xi \leq \arg a' \leq 2\pi-\tau +\xi.
\end{equation}
If there were no such $\d_0$,  $X$ would contain a point outside $\S$.
\medskip

Suppose the assertion in the lemma is false. Then, there exist
$\d\in(0,\d_0)$, $a\in \g_\e \cap S_\d,a$, a point  $a'\in\G$ with $|a-a'|=\e$, $b_1 \in S_{\d,a}^1$ and $b_2 \in S_{\d,a}^2$ such that $b_1,b_2 \notin B(a',\e)$.
Assume as we may that
\[ - (\tau - \cos^{-1}(\frac{\d}{2\e})) \leq \arg{b_1} < \arg{a} < \arg{b_2} \leq \tau - \cos^{-1}(\frac{\d}{2\e}).  \]
Let $l_1$ (resp. $l_2$) be the line that bisects the segment $[a,b_1]$ (resp. $[a,b_2]$). Since $|b_j- a'| \geq \e = |a-a'|$ for $j=1$ and $2$, the point $a'$ lies in the closure of the component of $\R^2\setminus \{l_1,l_2\}$ that contains $a$. In particular by \eqref{argument-S},
\[
- (\tau - \cos^{-1}(\frac{\d}{2\e})) \leq  \frac{\arg{b_1} + \arg{a}}{2} \leq \arg{a'} \leq \frac{\arg{b_2} + \arg{a}}{2}\leq \tau - \cos^{-1}(\frac{\d}{2\e}),
\]
which is impossible in view of (\ref{argument-a'-1}) and the fact that $\xi < \pi/24 < \cos^{-1}(\frac{\d_0}{2\e})$. This proves the second assertion and the lemma.
\end{proof}

\begin{lem}
\label{K-D}
Suppose that for some $\e\neq 0$, there exist a connected component $D$ of $\D_{\e}$ and a connected component $G$ of $\g_{\e}\cup\D_{\e}$
such that $\overline{D} \subsetneq G$.
Then, there exists a point $x_0 \in \partial D$ and points $x_1,x_2 \in \G$ such that $x_0,x_1,x_2$ are collinear and
\[ |x_0 - x_1| = |x_0 - x_2| = |\e|. \]
Furthermore, $\G^{\{x_0\}}=\{x_1,x_2\}$.
\end{lem}

\begin{proof}
We treat the case $\e>0$ only; the case $\e<0$  is similar.
Let $E = G \setminus \overline{D}$. Since $G$ is connected, we have that $\overline{E} \cap \overline{D} \neq \emptyset$ and $\overline{E} \cap \partial D \neq \emptyset$. Fix a point $x_0 \in \overline{E} \cap \partial D$; clearly $x_0$ is a non-isolated point in $ \g_\e$.
Define $X=\G^{\{x_0\}}$, the shortest subarc $\Sigma$ of $S(x_0,\e)$  containing $X$,  its end points $x_1,x_2$, the open set $U$, and the number
 $\d_0 >0 $,  relative to the point $x_0$ as  in Lemma \ref{4points}.

Suppose that $|x_1- x_2|< 2 \e$. Then for $\d\in(0,\e_0)$, $S(x_0,\d)\setminus U$ is the arc $S_\d$. Since $x_0 \in \partial D$, there exists a number $\d_1=\d_1(x_0,D,\e)>0$ such that  $ D \cap S_\d$ contains a non-trivial arc for every $0<\d<\d_1$.   Therefore
$\partial D \cap S_\d$ contains at least two points in $ \g_\e$.  Hence,
 by Lemma \ref{4points}, $E \cap S_\d =\emptyset$ when $0<\d<\min\{\d_0,\d_1\}$. This contradicts the assumption
$x_0 \in \overline{E}$. Therefore $|x_1- x_2|= 2 \e$ and $x_0, x_1$ and $ x_2$ are collinear.

We now prove $\G^{\{x_0\}}=\{x_1,x_2\}$.
Assume, as in Lemma \ref{4points}, that  $x_0 = 0$, $\Sigma \subset \{\text{Re}\, z \leq 0\}$, $x_1 = \e e^{i\pi/2}$ and $x_2 = \e e^{i 3\pi/2}$.
Suppose there exists another point $x_3 \in \G^{\{x_0\}}\setminus \{x_1,x_2\}$; so $\text{Re}\, x_3 < 0$. Observe, by elementary calculations, that there exists $\d_2=\d_2(x_3,\e)\in (0,\e_0)$ so that for any $y$ in the half disk $B(0,\d_2) \cap \{\text{Re}\, z < 0\}$, one of the numbers $|y-x_1|$, $|y-x_2|$, $|y-x_3|$ is strictly less than $\e$.
Therefore, $(\D_\e \cup \g_\e) \cap  B(0,\d_2 ) \subset  \{\text{Re}\, z \geq 0\}   \setminus U $.
Since $x_0\in \overline D$, $\partial D \cap S_\d$ contains at least two points in $ \g_\e$ for all sufficiently small $\d$. As before, it follows from Lemma \ref{4points} that  $E \cap S_\d$ must be empty for all sufficiently small $\d$, a contradiction. This proves that $\G^{\{x_0\}}=\{x_1,x_2\}$, and the lemma.
\end{proof}

The next two propositions lead naturally to the $(1/2,r_0)$-chordal condition for the LJC property in
 Theorem \ref{betathmLC}.

\begin{prop}\label{boundary}
Suppose that for some $\e\neq 0$, $\D_\e \neq \emptyset$,  $\g_{\e}\cup \D_{\e}$ is connected, and $\overline\D_\e \subsetneq \g_{\e}\cup \D_{\e} $. Then, there exist points $x_0\in \g_{\e}$ and $x_1,x_2 \in \G$ which are collinear such that
\[ |x_0-x_1| = |x_0-x_2| = |\e|. \]
Moreover, $\G^{\{x_0\}} = \{x_1,x_2\}$.
\end{prop}

From the assumptions, there exist a connected component $D$ of $\D_{\e}$ and a connected component $G$ of $\g_{\e}\cup\D_{\e}$
such that $D \subset \overline{D} \subsetneq G$. The proposition follows from Lemma \ref{K-D}.

\begin{rem}
The point $x_0$ in Proposition \ref{boundary}, which is chosen according to Lemma \ref{K-D}, lies, in fact, on the boundary of a component of $\D_\e$.
\end{rem}

\begin{prop}\label{gammadelta}
Suppose that $\D_{\e}\neq \emptyset$  and $\g_{\e}\cup \D_{\e}$ is not connected for some $\e\neq 0$. Then, there exist points $x_0\in \Omega$ and $x_1,x_2 \in \G$ which are collinear such that
\[ |x_0-x_1| = |x_0-x_2| = \rm{dist} (x_0,\G)<|\e|. \]
Moreover, $\G^{\{x_0\}} = \{x_1,x_2\}$.
\end{prop}

\begin{proof}

We first prove the proposition for $\e>0$. Choose  a connected component $D$  of $\D_\e$, a point $x\in D$, and a point $y$ in a connected component of $\D_{\e}\cup\g_{\e}$ that does not meet $\overline D$, and define
\[
d_0 = \sup\{\d>0 \colon x, y \,\, \text{are in a common component of} \,\, \D_\d\}.
\]
Since $\Omega$ is path connected, $d_0>0$; and since $x$ and $y$ are in two different components of the closed set $\g_{\e}\cup \D_{\e}$, $d_0 < \e$.

For $\d \in (0,d_0)$, let $G_\d$ be the component of $\D_\d$ that contains $x$ and $y$. Then, for $0 <\d<\d'<d_0$ we have $G_{\d'} \subset \overline{G_{\d'}} \subset G_\d$. Since $\{\overline{G_\d}\}_{\d\in (0,d_0)}$ is a nested family of compact connected sets, the intersection $G = \bigcap_{0 <\d<d_0} \overline G_\d$ is a connected subset of $\g_{d_0}\cup\D_{d_0}$ that contains $\overline D \cup \{y\}$.

We claim that $G$ is the component of $\g_{d_0}\cup \D_{d_0}$ that contains $x,y$. Indeed, let $\tilde G$ be the component of $\g_{d_0}\cup \D_{d_0}$ that contains $x,y$. Clearly $G\subset \tilde G$. Since  the set $\bigcup_{x\in \tilde G}B(x,\d)$ is open and connected for every $\d \in (0,d_0)$,
\[\tilde G \subset \bigcup_{x\in \tilde G}B(x,\d) \subset G_{d_0-\d} \qquad \text{for each }\,\,\d \in (0,d_0).
\]
So $\tilde G \subset G$ and $G$ is the component of $\g_{d_0}\cup \D_{d_0}$ that contains $x,y$. Hence, $\overline{D} \subsetneq G$ and
the proposition now follows from Lemma \ref{K-D}.

Suppose now that $\e<0$. Choose $D, x, y$ as before, and define
\[
d_0 = \inf\{\d<0 \colon x, y \,\, \text{are in a common component of} \,\, \D_\d\}.
\]
For $\d \in (d_0,0), $ let $G'_{\d}$ be the component of $\D_{\d}$ that contains $x$ and $y$, and let $G' = \bigcap_{d_0<\d<0 } \overline {G'}_\d$.  Since, for all $d_0<\d<0$,  sets $\R^2 \setminus \overline {G'}_\d$ are contained in a fixed planar disk, the intersection $G'$ is connected. The rest of the proof is similar to that of the case $\e>0$.
\end{proof}

\begin{rem}
The point $x_0$ in Proposition \ref{gammadelta}, chosen according to Lemma \ref{K-D}, lies on the boundary of a component of $\D_{d_0}$ for some $0< |d_0|<|\e|$.
\end{rem}

\section{Level Curves and Level Quasicircles}\label{mainresults}

In this section, we give the proofs of Theorem \ref{betathmLC} and Theorem \ref{betathmLQC} along with two examples that show the sharpness of the conditions.

\begin{proof}[Proof of Theorem \ref{betathmLC}]
We give the proof for the case $\e>0$; the case $\e<0$ is practically the same. By the assumption of the theorem, there exists $r_0>0$ such that $\z_{\G}(x,y)\leq 1/2$, for all $x,y \in \G$ with $|x-y|\leq r_0$.

First we claim that $\D_\e\cup\g_\e$  is connected for all  $\e\in (0,r_0/2)$. Otherwise, by Proposition, \ref{gammadelta}, there exist $d_0 \in (0,r_0/2)$  and collinear points $x_0 \in \g_{d_0}$ and $x_1,x_2 \in \G$ such that $\G^{\{x_0\}} =\{y\in \G \colon |x_0-y|=d_0\}= \{x_1,x_2\}$.
The line $l$ that contains $x_0$ and is perpendicular to $l_{x_1,x_2}$ intersects $\G(x_1,x_2)$ at some point $z$. Note that $|x_1-x_2|=2 d_0< r_0$ and that
\[ \dist(z, l_{x_1,x_2}) = |x_0-z| > \dist (x_0,\G)=d_0. \]
So $\z_{\G}(x_1,x_2) > 1/2$, a contradiction.

Next we claim that  $\D_\e$ must be connected for all $\e \in(0,r_0/2) $.
Otherwise, for some  $\e\in (0,r_0/2)$  the open set  $\D_\e$ would have at least two components, called $D_1, D_2$. By  the continuity of the distance function, each  $D_j, j=1, 2,$ would contain a point $z_j$ of distance $\e' $ to $\G$, for some $\e' \in (\e,r_0/2)$. This would imply  that $\D_{\e'}\cup \g_{\e'}$ is not connected; this contradicts the previous claim.

Therefore, by Lemma \ref{appr},  $\partial \D_\e$ is a Jordan curve for every $\e \in(0,r_0/2)$.
It remains to check that $\g_\e= \partial \D_\e$ for all $\e \in(0,r_0/2)$.
Suppose $\partial \D_\e \varsubsetneq \g_\e$ for some $\e \in(0,r_0/2)$. Then $\overline \D_\e \varsubsetneq \D_\e \cup \g_\e$. Therefore, by Proposition \ref{boundary}, we can find collinear points $x_0\in \g_\e$ and $x_1,x_2 \in \G$ such that $\G^{\{x_0\}} = \{x_1,x_2\}$. As before, this will lead to
the inequality $\z_{\G}(x_1,x_2) > 1/2$, a contradiction. So $\g_\e= \partial \D_\e$.

This completes the proof of the theorem.
\end{proof}

\begin{rem}\label{sharpLJC}
The $(1/2,r_0)$-chordal condition is sharp for the conclusion of Theorem \ref{betathmLC}.

\emph{We construct a chord-arc curve $\G$ with $\z_{\G} = \frac{1}{2}$ which satisfies
\begin{enumerate}
\item[(i)] There exist two sequences of points $\{x_n\},\{y_n\}$ on $\G$ such that $|x_n-y_n| \to 0$ and $\z_{\G}(x_n,y_n) = \frac{1}{2} + 2^{-n}$.
\item[(ii)] There exists a decreasing sequence of positive numbers $\{\e_n\}$ with $\e_n \to 0$ such that $\g_{\e_n}$ is not a Jordan curve.
\end{enumerate}
as follows. Let $\G$ be the boundary of the domain
\[ D = [-1,2] \times [-3,0] \cup \bigcup_{n=0}^{\infty} [2^{-n}-2^{-n-2} , 2^{-n}] \times [0, 2^{-n-2}(1/2 + 2^{-n})] .\]
Observe that $\G$ is a Jordan curve and it is not difficult to show that $\G$ is also a chord-arc. Set, for any $n \in \N$,
\[ x_n = (2^{-n} - 2^{-n-2} , 0) \text{ and } y_n = (2^{-n}, 0). \]
Note that $\z_{\G}(x_n,y_n) = \frac{1}{2} + 2^{-n}$ and that it is not hard to check that $\z_{\G} = \frac{1}{2}$. Let $\Lambda_n = \G(x_n,y_n)$ and $\e_n = 2^{-n-3}$. Then, the set $\g_{\e_n}^{\Lambda_n}=\{x\in \g_{\e_n}\colon \dist(x,\Lambda_n)=\e_n \}$ is the union of the line segment $\{x_n + 2^{-n-3}\}\times [0,2^{-2n-2}]$
and  two quarter-circles $\{x_n + \e_n e^{i\theta} \colon \frac{3\pi}{2}\leq \theta \leq 2 \pi \} \bigcup \{y_n + \e_n e^{i\theta} \colon \pi \leq \theta \leq  \frac{3\pi}{2}   \}$. It follows that $\g_{\e_n}$ is not a Jordan curve.}
\end{rem}

\medskip

We now apply Lemma \ref{curvesincomp}, Lemma \ref{quasiconvex}, and Theorem \ref{betathmLC} to prove Theorem \ref{betathmLQC}. Recall from Lemma \ref{curvesincomp} that $\D_\e$, if a Jordan domain, has no inward cusp. Condition $\z_\G <1/2$, together with the estimates \eqref{zeta-half} below, shows that $\D_\e$ \emph{has no outward cusps}.

\begin{proof}[Proof of Theorem \ref{betathmLQC}]
We prove the theorem for $\e>0$ only. The proof for the case $\e<0$ is practically the same.

By the assumption of the theorem, there exist  $\z\in (0,1/2)$ and $r_0>0$ such that $\z_{\G}(x,y)\leq \z $ for all $x,y \in \G$ with $|x-y|\leq r_0$.
From Theorem \ref{betathmLC} and its proof,  $\g_\e$ is a Jordan curve for every $\e\in (0,r_0/2)$; by Lemma \ref{zetabounded}, $\G$ is a $K(\z)$-quasicircle, therefore satisfies the 2-point condition \eqref{3pts} for some constant $C(\z)>1$. Constants below will depend only on $\z$.

We now prove that there exists $K'>1$ depending only on $\z$ such that $\g_\e$ is a $K'$-quasicircle for any
\[
0 <\e< \min\{\frac{r_0}{10},\frac{\diam \G}{20 C(\z)}\}.
\]
By the $2$-point condition, it suffices to prove that there exists $M>1$, depending only on $\z$, such that
\[ \diam{\g_\e ( x,y)} \leq M|x-y|\quad \text{for all}\,\,x,y \in \g_\e. \]
Given $x$ and $y$ in $ \g_\e$, choose $x',y' \in \G$ such that $|x-x'| = |y-y'| = \e;$
segments $[x,x']$ and $[y,y']$ do not meet except possibly at $x'$ and $y'$.
By Remark \ref{remappr}, there exists a curve $\tau_{x,y}$, with $\tau_{x,y}\setminus \{x,y\}\subset \D_\e$, that connects $x$ to $y$, and satisfies $|x-y|\leq \diam{\tau_{x,y}} \leq C_1(\z)|x-y|$ for some constant $C_1(\z)>1$. Consider the domain $D$  enclosed by the Jordan curve $[x,x'] \cup \G(x',y') \cup [y,y'] \cup \tau_{x,y}$. Let $\g_\e(x,y)^*$ be the component of $\g_\e \setminus \{x,y\}$ that is contained in $D$; note that $\g_\e(x,y)^*$ and $\g_\e(x,y)$ are not necessarily the same arc. It suffices to show that
\[
\diam{\g_\e ( x,y)^*} \simeq |x-y|.
\]
We consider four cases according to the ratios  $|x'-y'|/\e$ and  $|x-y|/\e$.

\emph{Case 1.} $|x'-y'| \geq 4(1-\z)\e$. $\,$ In this case,
$ |x'-y'| -2\e \leq |x-y| \leq |x'-y'| + 2\e $,
which implies
\[ \frac{1-2\z}{2-2\z}|x'-y'| \leq |x-y| \leq \frac{3-2\z}{2-2\z}|x'-y'|.\]
Since $0<\z <1/2$, $\diam{\tau_{x,y}} \simeq |x-y|$ and  $\G$ is a $K(\z)$-quasicircle, we have $ \diam D\simeq |x-y|$. Hence, $\diam{\g_\e ( x,y)^*} \simeq |x-y|$.

\emph{Case 2.} $x' = y'$. $\,$ In this case, $\g(x,y)^*=\g(x,y)$. By Lemma \ref{circgamma}, $\g_\e(x,y)$ is a subarc of $S(x',\e)$ of length at most $\pi \e$, hence $\diam{\g_\e(x,y)} = |x-y|$.

\emph{Case 3.} $0 < |x'-y'| < 4(1-\z)\e$ and $ |x-y| \ge \e (1-2\z)^2/10$. $\,$
Since $\diam D \simeq \e$ and $\g_\e(x,y)^* \subset D$, we have $\diam \g_\e(x,y)^*\simeq |x-y| \simeq \e$.

\emph{Case 4.} $0 < |x'-y'| < 4(1-\z)\e$ and $0< |x-y| < \e (1-2\z)^2/10$. $\,$
In view of Lemma \ref{quasiconvex} and Remark \ref{remappr}, we may assume that $\diam \tau_{x,y} \le 5|x-y| <\e/2$. It is easy to check that in this case $\g(x,y)^*=\g(x,y)$. However, there is no relation between $|x-y|$ and $|x'-y'|$, and $\diam D$ may be much bigger than $|x-y|$. We will construct a new domain $D'$ whose closure contains $\g_\e(x,y)$ and has $\diam D'\simeq |x-y|$.

First, let $R(x',y')$ be the  rectangular domain whose boundary has two sides  parallel to the line $l_{x',y'}$ of length $a=|x'-y'|$, and two other sides having mid-points $x'$ and $y'$ and of  length  $b=2(\e-\z|x'-y'|)$. Then define a domain
\[
U(x',y')=B(x',\e)\cup B(y',\e) \cup R(x',y').
\]
It is possible that $R(x',y')$ is contained in $B(x',\e)\cup B(y',\e)$ for some pairs $x'$ and $y'$.
Nevertheless, $\partial U(x',y')$ are $K''$-quasicircles for some  constant $K''>1$ depending only on $\z$, in particular not on $x'$ and $y'$. This observation follows from the inequalities: $0<\z <1/2$,
\begin{equation}\label{zeta-half}
0< a=|x'-y'| < 4(1-\z)\e, \,\text{and} \,0< \e (1-2\z)^{2} < \frac{b}{2} = \e-\z|x'-y'| < \e.
\end{equation}

Next, we claim that $U(x',y')\cap \overline \D_\e =\emptyset$.
Indeed, for any $z\in R(x',y')$ the line containing $z$ and perpendicular to $l_{x',y'}$
must intersect the arc $\G(x',y')$ at some point $z'$.
Note that $\dist(z,\G)\leq \dist(z,\G(x',y')) \leq |z-z'|\leq \dist(z,l_{x',y'}) +\dist(z',l_{x',y'}) < \frac{b}{2} +\z |x'-y'| =\e$.
Clearly, $\dist(z,\G)<\e$ for all $z\in B(x',\e)\cup B(y',\e)$.

Recall that $x \in \partial B(x',\e)\cap \partial U(x',y')$ and $y \in \partial B(y',\e)\cap \partial U(x',y')$.
Let $T_{x,y}$ be the subarc of  $\partial U(x',y')$ connecting $x$ to $y$  that has the smaller diameter. Then, $T_{x,y} \subset \R^2\setminus \D_\e$, and $\diam T_{x,y} \simeq |x-y|$ because  $\partial U(x',y')$ is a $K''$-quasicircle.

To summarize, $\D_\e$ is a Jordan domain, $x$ and $y$ are two points on $\partial \D_\e$, and $\tau_{x,y}$, $\g_\e(x,y)$,  and $T_{x,y}$ are arcs connecting $x$ to $y$, with
$\tau_{x,y}\setminus \{x,y\}\subset \D_\e$, $\g_\e(x,y)\subset \partial \D_\e$, and $T_{x,y} \subset \R^2\setminus \D_\e$.

Let $D'$ be the domain enclosed by the Jordan curve $\tau_{x,y}\cup T_{x,y}$. We claim that $\g_\e(x,y)$ is contained in $\overline {D'}$. Otherwise, $\tau_{x,y}$ would be contained in the closure of the domain $D''$ enclosed by the Jordan curve $\g_\e(x,y) \cup T_{x,y}$. By the connectedness of $\D_\e$, the entire $\D_\e$ would be contained in $D''$. A preliminary estimate of $\diam \g_\e(x,y)$ from the fact  $ \g_\e(x,y)\subset D$ shows that
\[
\diam \g_\e(x,y) \le  5 |x-y| +2\e + C(\z)|x'-y'| \leq 7 C(\z)\e.
\]
Therefore,
\[
\diam \D_\e \leq \diam D'' \le \diam \g_\e(x,y) + \diam U(x',y') \]
\[
\leq
7 C(\z)\e +4\e+ |x'-y'| \leq 15 C(\z)\e <\frac{3}{4} \diam \G <\diam \D_{\e},
 \]
a contradiction. So $\g_\e(x,y)\subset \overline {D'}$, and therefore
 \[\diam \g_\e(x,y) \leq \diam D' \leq \diam \tau_{x,y} +\diam T_{x,y}  \simeq |x-y|.\]

This completes the proof of $\diam \g_\e(x,y) \simeq |x-y|$ for Case 4, and the theorem.
\end{proof}

\begin{rem}\label{sharpLQC}
The condition $\z_\G< 1/2$ is sharp for the conclusion of Theorem \ref{betathmLQC}.

\emph{We first make an observation.
Given $\alpha \in [0,\pi/12]$,
let $\sigma$ be the circular arc $\{e^{i\theta}\colon  \alpha \leq \theta \leq \pi-\alpha\}$, and
$\G'$ be the infinite simple curve obtained by replacing the segment  $[e^{i\alpha}, e^{i (\pi-\alpha)}]$ on  $l_{e^{i\alpha}, e^{i (\pi-\alpha)}}$ by $\sigma$. The set of points below $\G'$ that have unit distance to $\G'$ is a simple arc $\g'$ consisting of two horizontal semi-infinite lines and two circular arcs $\tau_1$ and $\tau_2$, where $\tau_1$ is a subarc of the circle $S(e^{i\alpha},1)$ connecting $0$ and $-i+e^{i\alpha}$ , and $\tau_2$ is a subarc of the circle $S( e^{i (\pi-\alpha)},1)$ connecting $0$ and $-i+e^{i (\pi-\alpha)}$. Since $\tau_1$ and $\tau_2$ meet at an angle $2\alpha$, the arc $\g'$ is a $K(\alpha)$-quasiline with $K(\alpha) \to \infty$ as $\alpha \to 0$.}

\emph{Fix now  a decreasing sequence $\alpha_n$ converging to $0$ with $\alpha_1 = \pi/12$, and another sequence $\e_n=4^{-n-2}$. Let $p_n$ be the point having coordinates $(2^{-n}, -\e_n \sin \alpha_n)$ and  $\sigma_n$ be the subarc of $S(p_n,\e_n)$ above the real axis; and  let $\omega$ be the simple curve that has end points $-1$ and $1$ and is the union of circular arcs $\bigcup_{n\ge 1} \sigma_n$ and a countable number of horizontal segments in $[0,1]$.  Fix a large $N\in \N$, and let $P$ be the boundary of a regular $N$-polygon in the lower half-plane which has $[-1,1]$ as one of its edges. Let $\G$ be the Jordan curve obtained from $P$ by replacing  the edge $[-1,1]$ by
$\omega$.}

\emph{It is not hard to see that for sufficiently large $N$,  $\G$ is a $K$-quasicircle for some $K>1$ independent of $N$,  that
$\z_{\G}(x,y) < 1/2$  for all $x,y\in \G$ with $|x-y|\le 1/2$, and that $\z_{\G} =  1/2$.}

\emph{On the other hand, every level curve $\g_{\e_n}$ is a $K_n$-quasicircle which contains two circular arcs, with the same  curvature, meeting at an angle $2\alpha_n$. Since $\alpha_n\to 0$, $K_n$'s cannot have a uniform upper bound. So $\G$ does not satisfy the LQC property.}
\end{rem}

\section{Level Chord-Arc Property}\label{LCAresults}

In this section we give the proof of Theorem \ref{LCAmain}. We start by recalling a known fact: \emph{if a bounded starlike domain in $\R^2$ satisfies a strong interior cone property  then its boundary is a chord-arc curve.}

For $a \in (0,\pi), h>0$, $x\in \R^2$ and $v\in \mathbb S^1$, denote by
\[ \mathscr{C}_{a,h}(x,v) = \{z \in \R^2 \colon \cos(a/2)\, |z-x| \leq v\cdot (z-x) \leq h \}\]
the truncated cone with vertex $x$, direction $v$, height $h$ and aperture $a$.

Suppose that $U \subset \R^2$ is a bounded  \emph{starlike domain} with respect to a point $x_0\in U$, i.e., for every $x \in \partial U$ the line segment $[x_0,x]$ intersects $\partial U$ only at the point $x$.
Suppose in addition
$(U, x_0)$ satisfies the \emph{strong interior cone property}, i.e.,
there exist $a \in (0,\pi), h>0$ so that the truncated cone $\mathscr{C}_{a,h}(x,v_x)\setminus \{x\}$, in the direction $v_x = (x_0-x)/|x_0-x|$, is contained in $U$ for every $x \in \partial U$.
Assume from now on $x_0=0$, and set
\[ \rho = \max\{|x|\colon x \in \partial U\}.
\]
We obtain, by elementary geometry, positive constants $c_1=c_1(a, \frac{h}{\rho}), c_2=c_2(a), c_3=c_3(a)$ such that
\[
c_2\, |x-y| \leq |x- |x|\frac{y}{|y|}|\leq \, c_3 |x-y|, \,\, \text{for all}\,\, x, y \in \partial U \,\, \text{with}\,\, |\frac{x}{|x|}-\frac{y}{|y|}| \le c_1.
\]
Let  $\psi \colon \partial U \to \mathbb S^1$ be the map $ x\mapsto \frac{x}{|x|}$. Then $\rho \psi$ is $L$-bi-Lipschitz for some constant $L>1 $ depending only on $a$ and $h/\rho$. Therefore $\partial U$ is a $C$-chord-arc curve for some constant $C>1 $ depending only on $a$ and $h/\rho$.

\bigskip

Essential to our proof of Theorem \ref{LCAmain} is a lemma of Brown \cite{Brown} on sets of constant distance from a compact subset $A$  of $\R^2$. Recall from the Introduction that for a given $\e>0$, the $\e$\emph{-boundary} of $A$ is the set
\[ \partial_{\e}(A) = \{x \in \R^2 \colon \dist(x,A) = \e \}. \]
In Lemma 1 of his paper, Brown  proved that \emph{if $\e > \diam{A}$, then $\partial_{\e}(A)$ is the boundary of a starlike domain $U_\e$ with respect to any point $x_0 \in A$}. In fact,  whenever  $\e > 3 \diam{A}$, $(U_\e,x_0)$ also possesses
the strong interior cone property, namely,
 the cone $\mathscr{C}_{\frac{\pi}{3},\frac{\e}{3}}(x,(x_0-x)/|x_0-x|)\setminus \{x\}$ with vertex $x\in \partial_{\e}(A)$ is contained in $U_\e$. Since $2\e <\diam (\partial_{\e}(A)) < 3\e $, we have the following.

\begin{lem}\label{Brown-chord arc}
There is a universal constant $c_0>1$ for the following.
Suppose that $A$ is a compact subset of $\R^2$ and that $\e > 3 \diam{A}$. Then the $\e$\emph{-boundary} $\partial_{\e}(A) $ of $A$ is a $c_0$-chord arc curve.
\end{lem}
\medskip

We now apply Lemma \ref{Brown-chord arc} locally and repeatedly to prove Theorem \ref{LCAmain}.

\begin{proof}[Proof of Theorem \ref{LCAmain}]
We prove the theorem for the case $\e>0$ only. The case $\e<0$ is essentially the same.

For the necessity, we only need to check that $\G$ is a chord-arc curve.  By the LCA property,
there exist $L>1$, $n_0 \in \N$, and for each $n \geq n_0$, an $L$-bi-Lipschitz homeomorphism $f_n$ of $\R^2$
such that $f_n(\overline{\mathbb{B}^2}) = \overline{\D_{\frac{1}{n}}}$.
Since $f_n|\overline{\mathbb{B}^2}$ are equicontinuous, by Arzela-Ascoli, there is a subsequence $f_{k_n}|\overline{\mathbb{B}^2}$ which converges to a homeomorphism $f$.
It is not hard to see that $f$ is bi-Lipschitz and maps $\overline{\mathbb{B}^2}$ onto $\overline\Omega$. Therefore,  $\G=f(\partial {\mathbb{B}^2})$ is a chord-arc curve.

To show the sufficiency, we assume that $\G$ is a $C_1$-chord-arc curve, and that there exist $\e_0 >0$ and $K>1$ such that the Jordan curves $\g_{\e}$ are $K$-quasicircles  for all  $\e\in (0, \e_0]$. In the rest of the proof, constants are understood to depend on  $C_1$ and $K$ only, in particular independent of $\e$.

For $\e\in (0,\e_0]$ and for a closed subset $\l \subset \g_{\e}$, we set
\[\G^{\l} = \{y \in \G \colon |y-x| = \e \,\,\text{for some}\,\, x \in \l \}=\{y \in \G \colon \dist(y,\l) = \e \}. \]
In general, $\G^{\l} $ need not be connected, and there is no relation between the diameter of $\l$ and the diameter of $\G^{\l}$.

We prove now that $\g_{\e}$ is a chord-arc curve.  Since $\g_{\e}$ is a $K$-quasicircle, it suffices to check
\[
\ell (\l) \lesssim  \diam{\l}\quad \text{ for all  subarcs}\,\, \l \subset \g_{\e}.\]
 We consider three cases according to the diameter of $\G^{\l}$.

\emph{Case 1.} $\diam{\G^{\l}} \leq \e/10$.
Set
\[ \partial_{\e}(\G^{\l}) = \{x \in \R^2 \colon \dist(x,\G^{\l}) = \e \}.\]
After a moment of reflection, we see that $\l \subset \partial_{\e}(\G^{\l})$.
By Lemma \ref{Brown-chord arc}, there exists a universal constant $c_0 > 1$ such that, for any $x,y \in \partial_{\e}(\G^{\l})$,\[ \ell (\partial_{\e}(\G^{\l})(x,y)) \leq c_0|x-y|;  \]
recall that $\partial_{\e}(\G^{\l})(x,y)$ is the subarc of $\partial_{\e}(\G^{\l})$ connecting $x$ and $y$ that has the smaller diameter. We deduce from this the following
\[
\ell (\l) \leq c_0\diam{\l}.
\]

\medskip

To prepare for the next two cases, we take $\L$ to be the subarc of $\G$ that contains $\G^{\l}$ having the smallest diameter.
Subdivide $\L$ into subarcs $\L_1,\L_2,\dots,\L_N$ which have mutually disjoint interiors and satisfy the condition
\[
\e/100 \leq \diam{\L_n} < \e/10\,\, \text{ for all}\,\, n=1,\dots,N.
\]
Since $\G$ is a quasicircle, $\diam \L \simeq \diam{\G^{\l}}$; since $\G$ is a $C_1$-chord-arc curve $ N \e \simeq \diam{\L}$. So,
\[
N \simeq \e^{-1}\diam{\L} \simeq \e^{-1}\diam{\G^{\l}}.
\]
Set $\l_n = \g_\e^{\L_n}\cap \l$ for $n=1,\dots,N$. Again, after a moment of reflection, we see that $\l = \bigcup_{n=1}^N \l_n$. Recall, from  Lemma \ref{orientation}, that  $\g_\e^{\L_n}=\{x\in\g_\e \colon \dist(x,\L_n)=\e\}$ are arcs whenever they are nonempty, so $\l_n$ are subarcs of $\g_\e$.  Note however that some of $\{\l_n\}$ may overlap.
We now apply Lemma \ref{Brown-chord arc} to the $\e$-boundary
$\partial_{\e}(\L_n)$ of  $\L_n$. Since $\l_n$ is also a subarc of $\partial_{\e}(\L_n)$,
it follows, as in Case 1,  that
\[
\ell (\l_n) \leq c_0\diam{\l_n}  \lesssim \e.
\]

\emph{Case 2.} $\e/10 < \diam{\G^{\l}} \leq 10\e$. From the estimates above, we obtain
\[
\ell (\l) \leq \sum_{n=1}^N\ell (\l_n) \leq
\sum_{n=1}^Nc_0\diam{\l_n} \leq Nc_0\diam{\l} \simeq \diam{\l}.
\]
Note that in this case, diameter of $\l$ might be much smaller than $\e$.

\emph{Case 3.} $10\e < \diam{\G^{\l}}$. In this case, it is geometrically evident that
\[
 \diam{\G^{\l}} - 2\e \leq   \diam{\l} \leq \diam{\G^{\l}} + 2\e, \]
hence  $\diam{\l}\simeq \diam{\G^{\l}}$.
Therefore,
\[ \ell (\l) \leq \sum_{n=1}^N\ell (\l_n) \leq \sum_{n=1}^N c_0 \diam{\l_n}\lesssim
N \e \simeq \diam{\G^{\l}} \simeq \diam {\l}. \qedhere \]
\end{proof}

\begin{rem}
Suppose that $\G$ is a Jordan curve. The proof of the previous theorem shows that each level set $\g_\e$, with $\e \neq 0$, is contained in a finite union of $c_0$-chord-arc curves, and that if $\g_\e$ is a quasicircle with $\e \neq 0$, then it is a $C(\G, \e)$-chord-arc curve.

\end{rem}

\begin{rem}\label{LCAcor}
Suppose that $\G$ is a Jordan curve which satisfies a local $C$-chord-arc condition with $1< C < \sqrt{2}$. Then, $\G$ has the \emph{LCA} property.
\end{rem}

\begin{proof}
Suppose that for any $x,y \in \G$ with $|x -y|<r_0$ and any $z\in \G(x,y)$, we have
$|z-x| + |z-y| \leq \ell (\G(x,y)) \leq C|x-y|$. Then $\G(x,y)$ is contained in the closed region whose boundary is the ellipse having foci at the points $x, y$ and semi-minor  $\frac{1}{2}\sqrt{C^2-1}|x-y|$.
Since $C < \sqrt{2}$, we have $\z_{\G} \leq \frac{1}{2}\sqrt{C^2-1} <1/2$. By Theorem \ref{betathmLQC}, $\G$ has the LQC property; and by Theorem \ref{LCAmain}, $\G$ has the LCA property.
\end{proof}

\section{Examples from Rohde's Snowflakes}\label{snowflakes}
In \cite{Roh}, Rohde gives an intrinsic characterization of planar quasicircles.
He defines explicitly a family $\mathcal F$ of snowflake-type curves, then proceeds to prove
that every quasicircle in the plane is  bi-Lipschitz homeomorphic to a member of this family.

Each of Rohde's snowflakes  $\mathcal S$ is constructed as follows. Fix a number $p \in [\frac{1}{4},\frac{1}{2})$, and let $\mathcal S_1$ be the unit square. The
polygon $\mathcal S_{n+1}$ is constructed by replacing each of the $4^n$ edges of $\mathcal S_{n}$ by a rescaled and rotated copy of one of the only two polygonal arcs allowed in Figure \ref{fig:figure2}, in such a way that the polygonal regions are expanding. The curve $\mathcal S$ is obtained by taking the limit of $\mathcal S_{n}$, just as in the construction of the usual von Koch snowflake. Clearly every Rohde's snowflake is a quasicircle. The entire collection of Rohde's snowflakes, with all possible $p\in [\frac{1}{4},\frac{1}{2})$, forms the family $\mathcal F$.

\begin{thmnn}[{\cite[Theorem 1.1]{Roh}}]
A bounded Jordan curve $\Gamma$ is a quasicircle if and only if there exist a curve $\mathcal{S}\in \mathcal F$ and a bi-Lipschitz homeomorphism $f$ of $\mathbb{R}^2$ so that $\Gamma = f(\mathcal{S})$.
\end{thmnn}

\noindent
\begin{figure}[ht]
\includegraphics[scale=1.9]{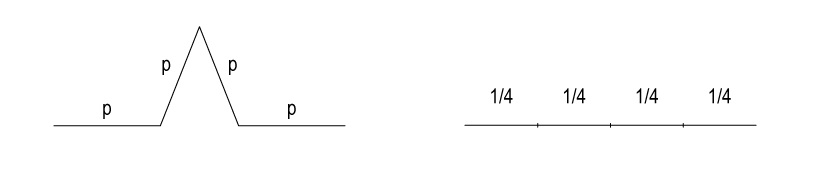}
\caption{}
\label{fig:figure1}
\end{figure}

Fix now a natural number $N\ge 4$. Suppose that a regular $N$-gon, of unit side length, is used in place of the unit square in the first step of Rohde's construction, while the remaining steps are unchanged.  So each snowflake-type curve is  the limit of a sequence of polygons, having  $N 4^{n-1}$ edges at the $n$-th stage. Let $\mathcal F_N$ be the family of these snowflakes. Then Rohde's argument shows that every quasicircle in $\R^2$ is the image of a curve in $\mathcal F_N$ under a bi-Lipschitz homeomorphism of $\R^2$.

Let $\mathcal F_{N,p}$ be the subfamily of curves in  $\mathcal F_N$ constructed using only the polygonal arcs indexed by $(1/4,1/4,1/4,1/4)$ and $(p,p,p,p)$.
It is not hard to see that there exist $N_0>4$ and $p_0 \in (\frac{1}{4},\frac{1}{2})$ for the following. Given $N\ge N_0$ and $1/4 \leq p\leq p_0$, there exists $0<\z_{N,p} <1/2$ and $r_{N,p}>0$ such that every curve $\mathcal S \in \mathcal F_{N,p} $ has the $(\z_{N,p},r_{N,p})$-chordal property, and therefore satisfies the LQC property.

\bibliographystyle{abbrv}
\bibliography{reference-levelsets}

\end{document}